\documentclass[12pt, a4paper, parskip=half, abstracton]{scrartcl}\usepackage{array}
\usepackage{marginnote}
\usepackage{xcolor}
\usepackage{amscd,amssymb,amsfonts,amsmath,latexsym,amsthm}
\usepackage{hyperref}
\usepackage[all,cmtip]{xy}
\usepackage{mathrsfs}
\usepackage{graphicx}
\usepackage{bm} 
\usepackage{etoolbox}
\def\hB{\hspace*{\fill}$\qed$}

\usepackage{chngcntr} 
\usepackage[nottoc]{tocbibind} 
\usepackage{defs_pp1}
\usepackage{slashed}
\usepackage[utf8]{inputenc}
\usepackage{microtype}
\usepackage[english]{babel}

\title{Coarse homotopy theory and boundary value problems}

\author{
Ulrich Bunke\thanks{Fakult{\"a}t f{\"u}r Mathematik,
Universit{\"a}t Regensburg,
93040 Regensburg,
GERMANY\newline
ulrich.bunke@mathematik.uni-regensburg.de} 
}

 \numberwithin{equation}{section}
\counterwithout{footnote}{section}

\newtheorem{theorem}{Theorem}[section] 
\newtheorem{prop}[theorem]{Proposition}
\newtheorem{lem}[theorem]{Lemma}
\newtheorem{ddd}[theorem]{Definition}
\newtheorem{kor}[theorem]{Corollary}
\newtheorem{ass}[theorem]{Assumption}

\theoremstyle{remark}
\theoremstyle{definition}

\newtheorem{ex}[theorem]{Example}
\newtheorem{rem}[theorem]{Remark}

\newcommand{\UBC}{\mathbf{UBC}}

\newcommand{\Yo}{\mathrm{Yo}}

\newcommand{\BC}{\mathbf{BornCoarse}}

\newcommand{\Cofib}{\mathrm{Cofib}}

\newcommand{\cO}{{\mathcal{O}}}

\newcommand{\susp}{\mathrm{susp}}

\renewcommand{\Dirac}{\slashed{D}}

\newcommand{\dist}{\mathrm{dist}}

\newcommand{\bdc}{\mathrm{bdc}}
\begin{document}

\maketitle

\begin{abstract}
We provide an interpretation of the APS index theorem of Piazza-Schick and Zeidler in terms of coarse homotopy theory.
On the one hand we propose a motivic version of the boundary value problem, the index theorem, and the associated secondary invariants. 
On the other hand, we discuss in detail how the abstract version is related with classical case for Dirac operators.

\end{abstract}

\tableofcontents

\section{Introduction}

The goal of this note is  a  motivic interpretation of the coarse index theory for manifolds with boundary
developed in Piazza-Schick \cite{MR3286895} and Zeidler \cite{MR3551834}.

In \cite{Bunke:2016aa} we proposed  a motivic approach to the  homotopy theory of bornological coarse spaces. Its equivariant generalization  is developed 
\cite{bekw}. As explained in \cite{ass} (in the non-equivariant case) one can  interpret the basic objects of index theory, namely the symbol of a Dirac operator and its index,  as coarse indices of   Dirac operators on appropriatly defined equivariant bornological coarse spaces.  
We will recall this coarse translation of index theory in Section \ref{ewfwefwfwefewf}.  
It allows to interpret the analytic assembly map, i.e., the transition
from the symbol to the index (which is usually of functional analytic nature), as a map induced by a map between  equivariant bornological coarse spaces.

In the present paper we provide proofs of the   index theorems  \cite[Thm. 1.22]{MR3286895} and
  \cite[Thm 6.5]{MR3551834}. Our proofs are not really different from
the ones given in the references (in particular in \cite{MR3551834}), but are completely rearranged. As a byproduct we obtain a generalization from the case of free actions as considered by Piazza-Schick and Zeidler to proper actions.

We have two main reasons for reconsidering the proofs of these
theorems. The first  is that we want to separate the non-formal aspects (really depending on the analysis of Dirac operators) from formal facts which are already true on a motivic level. Our second reason is that the abstract formulation of the index theory for boundary value problems potentially has other realizations.

In Section \ref{fiofwefewfewfewfewf} we provide an abstract version of an index theory for boundary value problems. This abstract version has no relation with Dirac operators at all. Following the philosophy  explained in  \cite{ass}, by just replacing equivariant coarse $K$-homology theory by an arbitrary equivariant coarse homology theory,  we can define the notions of   symbol  and index classes, and the analytic assembly map sending symbols to indices.
Furthermore, in the geometric situation of a boundary value problem, we define the notion of a boundary condition for a symbol and derive the corresponding index of the resulting abstract boundary value problem.
As a secondary invariant of the boundary condition we define an abstract version of the $\rho$-invariant, and we obtain an abstract version Theorem \ref{wfiowefewfew} of the APS-index theorem  \cite[Thm. 1.22]{MR3286895}. The abstract version 
of  \cite[Thm 6.5]{MR3551834} is true by the definition of the objects. We leave it to the future development of the subject whether this abstract theory of boundary value  problems has other realizations than the classical one for Dirac operators.

Our abstract index theory for boundary value problems depends on some purely geometric and motivic   constructions in    coarse  homotopy theory which we perform in Section \ref{fewzfefzewifzhiufzhiwefewf}. They are very much influenced by the argument for the proof of  \cite[Thm, 6.5]{MR3551834}, in particular the big commuting diagram appearing in the proof of that theorem. The analogue of this diagram 
 appears here as Proposition \ref{ergoierjoiregregeg}. 


In order to deduce the theorems on Dirac operators  \cite[Thm. 1.22]{MR3286895} and   \cite[Thm. 6.5]{MR3551834} from the abstract versions   we must
translate the abstract objects to the  classical ones in the analysis of Dirac operators. This is done in Section \ref{rgiorggergerge}.

It turns out that besides the definition of the coarse index (with support) $\Ind\cX(\Dirac_{M}, on(A))$ we only need two basic facts about this construction: its compatibility with suspension (Proposition \ref{feiwofwefwefewfw}) and locality (Theorem \ref{roigeorgergergerg}).  The details of the construction of the index class and the verification of these two properties are the contents of \cite{index}. 
 
\paragraph{Acknowledgements}
U.~Bunke  was supported by the SFB 1085 ``Higher Invariants''
funded by the Deutsche Forschungsgemeinschaft DFG.  
He thanks Thomas Schick for pointing out the relevance of 
 \cite[Thm, 6.5]{MR3551834}.

\section{A coarse construction}\label{fewzfefzewifzhiufzhiwefewf}

We consider a group $G$.  

In  \cite{bekw} we introduced categories of 
$G$-bornological coarse spaces $G\BC$ and $G$-uniform bornological coarse spaces $G\UBC$.  These categories have symmetric monoidal structures denoted by $\otimes$. 
We further introduced the notions of equivariant coarsely or uniformly excisive decompositions. 

We have a forgetful functor $G\UBC\to G\BC$ which we often implicity apply in order to evaluate a functor defined on $G\BC$ on objects of $G\UBC$.

We further defined a universal equivariant coarse homology theory $$\Yo^{s}:G\BC\to G\Sp\cX$$
whose target is the presentable stable $\infty$-category of equivariant coarse motivic spectra. If $X$ is a $G$-bornological coarse space, then $\Yo^{s}(X)$ is called the motive of $X$.

 We now introduce the basic object of the present paper.
  \begin{enumerate} \item  We consider a  $G$-uniform bornological coarse space $W_{\infty}$. \item  We assume an   equivariant   decomposition
$(M_{\infty},W)$ of $W_{\infty}$. 
\end{enumerate}
 
 We set $M:=W\cap M_{\infty}$ and let $f:M\to W$ denote the inclusion morphism.
\begin{ass}\label{fiofojwefwefwefwefwefwe}\mbox{}\begin{enumerate} \item We assume that the decomposition $(M_{\infty},W)$ of  $W_{\infty}$  
is coarsely and uniformly excisive. 
 \item We assume that    $M_{\infty}\cong [0,\infty)\otimes M$. \footnote{For most of the paper we only need a coarse equivalence, or weaker, flasqueness. The product form is only used in Prop. \ref{fwiowoefwfewfwf}}
 \item
We assume that the inclusion $f:M\to W$ is a coarse equivalence.\end{enumerate}
\end{ass}

 \begin{rem} The Assumption \ref{fiofojwefwefwefwefwefwe} abstracts the situation of a Riemannian  $G$-manifold $W$ with boundary $M$ with product collar such $\dist(w,M)$ is uniformly bounded for $w$ running in $W$. 
\hB\end{rem}
 \begin{rem}
 In this remark we recall the notation related to the cone construction. We refer to \cite[Sec. 9]{bekw} and \cite[Sec. 8]{ass}
 for more details. 
 
 For a $G$-uniform bornological space $X$ we consider the $G$-bornological coarse space  $$\cO(X )_{-}:=(\R\otimes X)_{h}\ .$$ The subscript $h$ indicates that we replace the original coarse structure on $\R\otimes X$ by the hybrid coarse structure which in this case 
 is induced by the uniform structure and the equivariant big family $((-\infty,n]\times X)_{n\in \nat}$.
 
 \begin{ddd}
 The subset $[0,\infty)\times X$ of $\cO(X )_{-}$ with the induced $G$-bornological coarse structure is called the cone $\cO(X)$ over $X$.
 \end{ddd}

 We have a natural inclusion $X\cong \{0\}\times X\to \cO(X)$ of bornological coarse spaces.
 
 \begin{ddd}
 The equivariant coarse motivic spectrum $$\cO^{\infty}(X):=\Cofib\big(\Yo^{s}(X)\to \Yo^{s}(\cO(X))\big)$$ is called the germs-at-$\infty$ object of the cone over $X$.
 \end{ddd}

By construction we have a fibre sequence in $G\Sp\cX$ \begin{equation}\label{fgoierjgioergegergerge}
\Yo^{s}(X)\to \Yo^{s}(\cO(X))\to \cO^{\infty}(X)\stackrel{\partial^{cone}}{\to} \Sigma\Yo^{s}(X)
\end{equation}
 which is called the cone sequence.
 
If $(A,B)$ is an equivariant  decomposition of a $G$-uniform bornological coarse space  $X$, then we consider the commuting square
$$\xymatrix{A\cap B\ar[r]\ar[d]&A\ar[d]\\B\ar[r]&X}$$
 of $G$-uniform bornological coarse spaces, where the subspaces $A$, $B$ and $A\cap B$  of $X$ are equipped with the induced structures.
 If $(A,B)$ is uniformly excisive, then $\cO^{\infty}$ sends this square to a push-out square in $G\Sp\cX$. If $(A,B)$ is in addition coarsely excisive, then again
 $\Yo^{s}\circ \cO$ sends  this square to a push-out square in $G\Sp\cX$.

One of the axioms of an equivariant coarse homology theory is that it vanishes on flasques. The cone functor $\cO$ sends flasque $G$-uniform bornological coarse spaces to weakly flasque $G$-bornological coarse spaces, but not to flasques, in general.
Most of the examples of equivariant coarse homology theories (e.g., equivariant coarse $K$-homology $K\cX^{G}$)  annihilate 
weakly flasque $G$-bornological coarse spaces. Such equivariant coarse homology theories are called strong.
In 
\cite[Sec. 4.4]{bekw} we introduce the universal strong equivariant coarse homology theory
$$\Yo^{s}_{wfl}:G\BC\to G\Sp\cX_{wfl}\ .$$
 
If $X$ is  a flasque $G$-uniform bornological coarse space, then the motives  $\Yo^{s}_{wfl}(X)$, $\Yo^{s}_{wfl}(\cO(X))$ and $\cO^{\infty}_{wfl}(X)$ in $G\Sp\cX_{wfl}$ vanish.

We have a sequence  of maps  of $G$-bornological coarse spaces $$X\to \cO(X)\to \cO(X)_{-}\stackrel{\partial^{geom}}{\to} \R\otimes X\ ,$$
 where the first two maps are given by the inclusions and the morphism
 $\partial^{geom}$ is the given by the identity of underlying sets. The sequence of equivariant coarse motivic spectra  obtained by applying $\Yo^{s}$ is equivalent to the cone sequence. Indeed, we have a commuting diagram 
  \begin{equation}\label{fiuhz98ifwref}\xymatrix{
\Yo^{s}(X)\ar@{=}[d]\ar[r]&\Yo^{s}(\cO(X))\ar[r]\ar@{=}[d]&\Yo^{s}(\cO(X)_{-}) \ar[r]^{\partial^{geom}}&\Yo^{s}(\R\otimes X)\ar[d]_{\susp}^{\simeq}\\
\Yo^{s}(X)\ar[r] &\Yo^{s}(\cO(X))\ar[r]&\cO^{\infty}(X)\ar[r]^{\partial^{cone}}\ar@{..>}[u]^{c}_{\simeq}&\Sigma \Yo^{s}(X)}\ .
\end{equation}
The dotted arrow is the one induced by the left commuting square and the universal property of $\cO^{\infty}(X)$ as a cofibre.
The suspension equivalence $\susp$ can be viewed as the boundary map of the Mayer-Vietoris sequence associated to the equivariant coarsely excisive decomposition $((-\infty,0]\times X,[0,\infty)\times X)$ of $\R\otimes X$.
  \hB
 \end{rem}

\begin{rem}
If $Y$ is an invariant  subset of a $G$-bornological coarse space $X$, then we write $Y_{X}$ in order stress that $Y$ is equipped with the induced structures from $X$.\hB
\end{rem}

We consider the invariant subset
$$Y:=(-\infty,0]\times W\cup [0,\infty)\times W_{\infty}$$ of $\R\times W_{\infty}$.
 Note that $\R\times W_{\infty}$ is also the underlying $G$-set of $\cO(W_{\infty})_{-}$.

 \begin{prop}\label{ergoierjoiregregeg}We have a commuting diagram in $G\Sp\cX$: $$\xymatrix{\cO^{\infty}(W_{\infty}) &&\\\Yo^{s}(\cO(W_{\infty}))\ar[u]_{\simeq}\ar[r]^{s}\ar[d]^{\partial^{\prime\prime}}&\Yo^{s}(Y_{\cO(W_{\infty})_{-}})\ar[d]^{\partial^{\prime}}\ar[ul]_{i}\ar[r]^{\partial}&\Sigma \Yo^{s}(W)\\
\Sigma \Yo^{s}(\cO(M))\ar@{=}[r]&\Sigma \Yo^{s}(\cO(M))&\Sigma \Yo^{s}(M)\ar[u]_{\simeq}\ar[l]\ar[ul]_{j}}\ .$$
 \end{prop}
\begin{proof}
The morphisms will be explained during the proof.

  We define the invariant subsets $$Y_{-}:=(-\infty,0]\times W \quad \mbox{and}\quad Y_{+}:= [0,\infty)\times W_{\infty}$$ of $Y$. These two subsets form an 
 invariant coarsely excisive decomposition $(Y_{-},Y_{+})$ of $Y_{\cO(W_{\infty})_{-}}$.  Since,  with the  induced structures,  $Y_{-}$ is flasque, $Y_{+}\cong \cO(W_{\infty})$ and $Y_{-}\cap Y_{+} \cong W$,  we have a Mayer-Vietoris sequence which is the middle sequence in the following diagram: \begin{equation}\label{f34oij3o4if34f34f3}
\xymatrix{&\cO^{\infty}(W_{\infty})&& \\ \Yo^{s}(W)\ar[r]&\Yo^{s}(\cO(W_{\infty}))\ar[u]_{\simeq}\ar[r]^{s}&\Yo^{s}(Y_{\cO(W_{\infty})_{-}})\ar[r]^{\partial}\ar[ul]_{i}&\Sigma \Yo^{s}(W)\\&& &\Sigma \Yo^{s}(M)\ar[u]^{f}_{\simeq}\ar[ul]_{j}}\ .
\end{equation}
The map $i$ is induced by the inclusion of   $Y$ into $\cO(W_{\infty})_{-}$ and the inverse of the equivalence $$c: \cO^{\infty}(W_{\infty})\stackrel{\simeq}{\to}\Yo^{s}(\cO(W_{\infty})_{-})$$ given by \eqref{fiuhz98ifwref}. The lower right vertical map  an equivalence since the inclusion $f:M\to W$ is a coarse equivalence by assumption.  
The left vertical map is the canonical morphism $\Yo^{s}(\cO(W_{\infty}))\to  \cO^{\infty}(W_{\infty})$ appearing in the cone sequence \eqref{fgoierjgioergegergerge}. Since
the inclusion $M_{\infty}\to W_{\infty}$ is a coarse equivalence and
$M_{\infty}$ is flasque we have $\Yo^{s}(W_{\infty})\simeq 0$. 
It follows from the cone sequence
 \eqref{fgoierjgioergegergerge}  that the left vertical map in \eqref{f34oij3o4if34f34f3} is an equivalence.
 The upper left triangle commutes by  construction.

In order to define the map $j$ we consider the invariant subsets  $$M_{-}:=(-\infty,0]\times M\quad \mbox{and}\quad  M_{+}:=[0,\infty)\times M$$ of $\R\times M$.
We define a map
$\tilde j:\R\otimes M\to Y$ such that $M_{-}  \to Y_{-}$ is induced by the embedding of  $M$ into $W$, and $M_{+}  \to  Y_{+}$ is the composition
$$M_{+}\cong M_{\infty}\to W_{\infty}\cong \{0\}\times W_{\infty}\to Y_{+}\ .$$
This map  $\tilde j$ is a morphism of $G$-bornological coarse spaces    and the induced map   $M\cong M_{-}\cap M_{+}\to Y_{-}\cap Y_{+}\cong W$ coincides with $f$. 
The map $j$ is then defined as the composition
$$j:\Sigma \Yo^{s}(M)\stackrel{\susp^{-1}, \simeq}{\to} \Yo^{s}(\R\otimes M)\stackrel{\tilde j}{\to} \Yo^{s}(Y)\ .$$
  We now turn the lower right triangle in \eqref{f34oij3o4if34f34f3}  into a commuting one.
To this end we observe that 
 we  have a map of Mayer-Vietoris sequences associated to the map of triples
 $$(\R\otimes M,M_{-},M_{+})\to (Y_{\cO(W_{\infty})_{-}},Y_{-},Y_{+})$$ induced by $\tilde j$:
$$\xymatrix{\Yo^{s}(M_{-})\oplus \Yo^{s}(M_{+})\ar[d]\ar[r] &\Yo^{s}(\R\otimes M)\ar[d]^{\tilde j}\ar[r]_{\simeq}^{\susp} &\Sigma\Yo^{s}(M)\ar@{..>}[dl]^{j}\ar[r]\ar[d]^{\simeq} &\Sigma \Yo^{s}(M_{-})\oplus\Sigma  \Yo^{s}(M_{+})\ar[d]\\ \Yo^{s}(Y_{-})\oplus \Yo^{s}(Y_{+})\ar[r]&\Yo^{s}(Y_{\cO(W_{\infty})_{-}})\ar[r]^{\partial}&\Sigma \Yo^{s}(W)\ar[r]&\Sigma \Yo^{s}(Y_{-})\oplus \Sigma \Yo^{s}(Y_{-})}\ .$$
Note that $M_{-}$ and $M_{+}$ are flasque. This   again explains again why the suspension morphism is an equivalence. The upper middle triangle commutes by the definition of $j$. Since the middle square commutes (i.e., has a prefered filler) we get the filler for the lower middle triangle. 
 
\begin{rem}\label{fiwefuweofuewf98ewfewfewfef}
For later use we note the following fact. The  diagram \eqref{f34oij3o4if34f34f3}  detemines an equivalence  \begin{equation}\label{gpoijopietrg345t}
s+j:\Yo^{s}(\cO(W_{\infty}))\oplus \Sigma \Yo^{s}(M)\simeq  \Yo^{s}(Y_{\cO(W_{\infty})_{-}})   \ .
\end{equation}\hB
 \end{rem}

We now consider the map of Mayer-Vietoris sequences induced by the map of triples
$$(\cO(W_{\infty}),[0,\infty)\times W, [0,\infty)\times M_{\infty})\to (Y_{\cO(W_{\infty})_{-}},\R\times W,[0,\infty)\times M_{\infty})\ .$$ It induces the commuting lower left square in the diagram extending \ref{f34oij3o4if34f34f3}
$$\xymatrix{\cO^{\infty}(W_{\infty}) &&\\\Yo^{s}(\cO(W_{\infty}))\ar[u]_{\simeq}\ar[r]^{s}\ar[d]^{\partial^{\prime\prime}}&\Yo^{s}(Y_{\cO(W_{\infty})_{-}})\ar[d]^{\partial^{\prime}}\ar[ul]_{i}\ar[r]^{\partial}&\Sigma \Yo^{s}(W)\\
\Sigma \Yo^{s}(\cO(M))\ar@{=}[r]&\Sigma \Yo^{s}(\cO(M))&\Sigma \Yo^{s}(M)\ar[u]_{\simeq}\ar[l]\ar[ul]_{j}}\ .$$
The lower right arrow is induced by the embedding of $M$ into the cone over $M$.  
We must show that the lower triangle in the lower right square commutes.
To this end we consider the map of Mayer-Vietoris sequences associated to the map of triples
$$(\R\otimes M,M_{-},M_{+})\to (Y,\R\times W,[0,\infty)\times M_{\infty})$$ again
induced by $\tilde j$.
We get a commuting square
$$\xymatrix{\Yo^{s}(\R\otimes M)\ar[d]^{\tilde j}\ar[r]^{\susp}_{\simeq}&\Sigma \Yo^{s}(M)\ar@{..>}[dl]_{j}\ar[d]\\\Yo^{s}(Y_{\cO(W_{\infty})_{-}})\ar[r]^{\partial^{\prime}}&\Sigma \Yo^{s}(\cO(M))}\ .$$
The upper triangle commutes by the definition of $j$. Hence the lower triangle commutes.\end{proof}
 
 We now consider the following diagram  \begin{equation}\label{goijoregerg}\hspace{-1cm}
\xymatrix{\Yo^{s}(\cO(M_{-}))\ar[r]&\Yo^{s}(\cO(\R\otimes M))\ar@/^2.5cm/@{..>}[drr]_(0.4){\susp} \ar[r]^(0.4){!} &\Yo^{s}(\cO(\R\otimes M))/\Yo^{s}(\cO(M_{-})) &\\
\Yo^{s} (\cO(M))\ar[d]\ar[u]\ar[r]&\Yo^{s}(\cO(M_{+})\ar[d] \ar[u]\ar[r]&\Yo^{s}(\cO(M_{+}))/\Yo^{s}(\cO(M))\ar[u]_{\simeq}\ar[r]^(0.6){!} \ar[d]^{\simeq}&\Sigma \Yo^{s}(\cO(M) ) \\\Yo^{s}((\R\times W)_{\cO(W_{\infty})_{-}})\ar[r]  &\Yo^{s}(Y_{\cO(W_{\infty})_{-}})\ar@/_2.5cm/@{..>}[urr]^(0.4){\partial^{\prime}}\ar[r]&\Yo^{s}(Y_{\cO(W_{\infty})_{-}})/\Yo^{s}((\R\times W)_{\cO(W_{\infty})_{-}}) & }
\end{equation}
The two  left squares are push-out squares
associated to the decompositions
$(M_{-},M_{+})$ of $\R\times M$ and
$(\R\times W,\cO(M_{+}))$ of $Y_{\cO(W_{\infty})_{-}}$.
This explains the vertical equivalences in the right.
The commutativity of the upper part reflects the definition of the suspension equivalence $\susp$,
and the lower part gives the definition $\partial^{\prime}$.

We call the equivalence
$$ \Yo^{s} (\cO(\R\otimes M ))/\Yo^{s}(\cO(M_{-})) \simeq 
\Yo^{s}(Y_{\cO(W_{\infty})_{-}})/\Yo^{s} ((\R\times W)_{\cO(W_{\infty)_{-}}} )$$
the excision equivalence.

 \begin{prop}\label{fwiowoefwfewfwf}
 After going over to weakly flasque motivic spectra $G\Sp\cX_{wfl}$ the marked morphisms in \eqref{goijoregerg} become equivalences.
 \end{prop}
\begin{proof}
We note that $M_{-}$ and $M_{+}$ are flasque $G$-uniform bornological coarse spaces. Therefore $\Yo^{s}_{wfl}(\cO(M_{- }))\simeq 0$ and
$\Yo^{s}_{wfl}(\cO(M_{+}))\simeq 0$.
The diagram simplifies to
 \begin{equation}\label{goijoreddgerg}\hspace{-2cm}
\xymatrix{0\ar[r]&\Yo_{wfl}^{s}(\cO(\R\otimes M))\ar@/^2.5cm/@{..>}[drr]_(0.4){\delta^{\prime}} \ar[r]^(0.4){!\simeq } &\Yo^{s}_{wfl}(\cO(\R\otimes M))  &\\
\Yo^{s}_{wfl}(\cO(M))\ar[d]\ar[u]\ar[r]&0\ar[d] \ar[u]\ar[r]&0/\Yo^{s}_{wfl}(\cO(M))\ar[u]_{\simeq}\ar[r]^(0.6){! \simeq} \ar[d]^{\simeq}&\Sigma \Yo^{s}_{wfl}(\cO(M) ) \\\Yo^{s}_{wfl}((\R\times W)_{\cO(W_{\infty})_{-}})\ar[r]  &\Yo^{s}_{wfl}(Y_{\cO(W_{\infty})_{-}})\ar@/_2.5cm/@{..>}[urr]^(0.4){\partial^{\prime}}\ar[r]&\Yo^{s}_{wfl}(Y_{\cO(W_{\infty})_{-}})/\Yo^{s}_{wfl}((\R\times W)_{\cO(W_{\infty})_{-}}) & }
\end{equation}
\end{proof}

%
%

\section{Abstract boundary value problems}\label{fiofwefewfewfewfewf}

We fix a spectrum-valued     equivariant coarse homology theory   $E$. If $X$ is a $G$-bornological coarse space, then for every integer $n$ we let
$E_{n}(X)$ denote the $n$'th stable homotopy group of the spectrum $E(X)$.

Let $X$ be a $G$-uniform bornological coarse space and $n$ be an integer.
\begin{ddd}
A class $\sigma$ in $  E_{n+1}(\cO^{\infty}(X))$
is called a symbol of degree $n$ on $X$
\end{ddd}

Let $\sigma$ be a symbol of degree $n$ on $X$.

\begin{ddd}\label{fwiowfwefewfwfw}
The class $\Ind(\sigma):=\partial^{cone}(\sigma)$ in $  E_{n}(X)$
is called the index of $\sigma$.
\end{ddd}

\begin{rem}
In Section \ref{ewfwefwfwefewf} we will motivate these definitions by relating them to the classical notions of  usual index theory. \hB
\end{rem}

We now retain the notions and assumptions introduced in Section \ref{fewzfefzewifzhiufzhiwefewf}.

Let $\sigma$ be a symbol of degree $n$ on $W_{\infty}$. Then  its index
$\Ind(\sigma)$ is uninteresting since $W_{\infty}$ is flasque and hence
$ E_{n}(W_{\infty})\simeq 0$. 

In ordinary index theory, in this situation one usually adds boundary conditions and defines an index in $ E_{n}(W)$.
Recall the map $i$ from the diagram in Proposition \ref{ergoierjoiregregeg}.
 
Let $\sigma$ be a symbol of degree $n$ on $W_{\infty}$.

\begin{ddd}\label{oriergegegergg}
A boundary condition for $\sigma$ is a choice of a class
$$\bdc(\sigma, W)   \:\:\mbox{in}\:\: E_{n+1}(Y_{\cO(W_{\infty})_{-}})$$
such that $i(\bdc(\sigma,W)) =\sigma$.
\end{ddd}

Let $\partial$ be the map as in  Proposition \ref{ergoierjoiregregeg}.
Let furthermore  $\sigma$ be a symbol of degree $n$ on $W_{\infty}$ and $\bdc(\sigma, W) $ be a  boundary condition for $\sigma$.

\begin{ddd}\label{wiwogrgregregerg} The index of $\sigma$ subject to the boundary condition $\bdc(\sigma, W)$ is 
the class
$$\Ind(\sigma, on (W)):=\partial \bdc(\sigma, W)  \:\:\mbox{in}\:\: E_{n}(W)\ .$$
 \end{ddd}

 Using Remark \ref{fiwefuweofuewf98ewfewfewfef} we get uniquely determined classes
 $$\sigma_{1}  \:\:\mbox{in}\:\: E_{n+1}(\cO(W_{\infty}))\quad \mbox{and} \quad \sigma_{2}  \:\:\mbox{in}\:\: E_{n}(M) $$
 such that
 $$s(\sigma_{1})+j(\sigma_{2})=\bdc(\sigma, W) \ .$$
 Then
 $$\sigma=i(\sigma_{1})\ , \quad \Ind(\sigma, on (W))=\partial( \sigma_{2})\ .$$
It follows that the set of choices of  boundary conditions for the symbol $\sigma$ is in bijection with the set of choices for the elements  $\sigma_{2}$, i.e, with the set $E_{n}(M)$.
Furthermore, in view of the equivalence $f:E(M)\to E(W)$, every prescribed index of the boundary value problem can be achieved  by choosing $\sigma_{2}$ appropriately. In fact, $\sigma_{2}$ is determined uniquely by the index.

Let $\partial^{\prime}$ be the map as in   Proposition \ref{ergoierjoiregregeg}.
Let furthermore  $\bdc(\sigma, W) $ be a  boundary condition for the symbol  $\sigma$  of degree $n$.

\begin{ddd}\label{fiowefoefewfewfewf}
The class $$\rho(M):=\partial^{\prime} (\bdc(\sigma,W))\in E_{n}(\cO(M))$$
is called the $\rho$-invariant of the boundary condition.
\end{ddd}

Note that the $\rho$-invariant depends on the choice of the boundary condition.
Let \begin{equation}\label{hfi83i9zfh}
\phi:=\cO(f):\cO(M)\to \cO(W)\ , \quad \psi: W\to \cO(W)
\end{equation}
be the canonical maps.
The following theorem is an abstract version of the the Atiyah-Patodi-Singer index theorem of Piazza-Schick.
\begin{theorem}[Abstract APS-Theorem]\label{wfiowefewfew}
We have the equality
$$\phi(\rho(M))=\psi(\Ind(\sigma, on (W)))\:\:\mbox{in}\:\: E_{n}(\cO(W)) \ .$$
\end{theorem}
\proof
We have (using the notation in the diagram of Proposition \ref{ergoierjoiregregeg})
\begin{eqnarray*}
\phi(\rho(M))&=&\phi(\partial^{\prime}(s(\sigma_{1})))+\phi(\partial^{\prime}(j(\sigma_{2})))\\&=&
\phi(\partial^{\prime\prime}(\sigma_{1}))+\psi(\Ind(\sigma, in(W)))\\
&=&\psi(\Ind(\sigma, in(W)))\ .
\end{eqnarray*}
For the first equality we use the commutativity of the diagram 
given in Proposition \ref{ergoierjoiregregeg}.
The composition $\phi\circ \partial^{\prime\prime}$ vanishes by the exactness of the Mayer-Vietoris sequence defining $ \partial^{\prime\prime}$. \hB

\begin{rem}
In Section \ref{rgiorggergerge} we will motivate the abstract definitions by relating them with the classical constructions in the case of coarse index theory of Dirac operators.  \hB
\end{rem}

\section{A coarse view on index theory of Dirac operators}\label{ewfwefwfwefewf}

In this section we describe the transition from the analysis of Dirac operators to homotopy theory. The basic object  is the  coarse index class $\Ind\cX(\Dirac_{M}, on(A))$  in $K\cX^{G}_{n}(\{A\})$ of a Dirac operator $\Dirac_{M}$ of degree $n$ which is uniformly locally positive outside of an invariant subset $A$ of $M$.  


We need two basic non-formal (i.e., really depending on analytic facts about Dirac operators) properties of this index class:
\begin{enumerate}
\item compatibility with suspension:   
Proposition \ref{feiwofwefwefewfw}.
\item  locality:     Coarse Relative Index Theorem \ref{roigeorgergergerg}
\end{enumerate}

In Section \ref{rgiorggergerge} we explain how the abstract index theory of Section \ref{fiofwefewfewfewfewf} specializes to the usual index theory on manifolds with cylindrical ends as developed by Piazza-Schick \cite{MR3286895} and Zeidler \cite{MR3551834}.

\subsection{The coarse index class with support}

In the following graded means $\Z/2\Z$-graded.
Let $n$ be a non-negative integer and consider the graded complex Clifford algebra $\Cl^{n}$ with $n$ generators. 
We consider a group $G$ and a complete Riemannian manifold $M$ with a proper action of $G$ by isometries. 

An equivariant Dirac operator of degree $n$ is a Dirac operator $\Dirac_{M}$ associated to a  graded equivariant Dirac bundle $E\to M$ with an additional fibrewise right $\Cl^{n}$-action. We require that this $\Cl^{n}$-action commutes with the Clifford multiplication by tangent vectors and is preserved by the connection.

%

  The Weizenboeck formula $$\Dirac_{M}^{2}:=\nabla^{*}\nabla+R$$ for $\Dirac_{M}$ 
  expresses the square of the Dirac operator as a sum of the Bochner Laplacian associated to the connection $\nabla$ of $E$ and a selfadjoint bundle endomorphism $R$ in $C^{\infty}(M,\End(E)^{sa})$.
 
 Let $B$ be a subset of $M$.
 \begin{ddd}
 $\Dirac_{M}$ is uniformly locally positive on $B$ if there exists a positive real number $c$  such that for every $x$ in $B$ we have
 $c \:\id_{E_{x}}\le R (x)$.   \end{ddd}

 \begin{ex} \label{gioregrgergreg}
 Assume that $M$ is a Riemannian  manifold of dimension $n$ with an isometric action of $G$. We assume that $M$ has a spin structure, and that  the action of $G$ lifts to the corresponding  $Spin(n)$-principal bundle $P\to M$. The group $Spin(n)$ can be considered as a subgroup of the even units of $\Cl^{n}$ and acts on $\Cl^{n}$ by left multiplication.  We  define the graded complex vector bundle $E^{spin}:=P\times_{Spin(n)}\Cl^{n}$. It is a
  $G$-equivariant bundle of right $\Cl^{n}$-modules. 
 
 Let  $(e_{i})_{i=1,\dots,n}$  be the generators of $\Cl^{n}$. If we   declare the products $ e_{i_{1}}\dots e_{i_{k}}$ in $\Cl^{n}$ for all $1\le i_{1}<\dots <i_{k}\le n$ to be orthonormal, then we obtain a hermitean $Spin(n)$-invariant scalar product on $\Cl^{n}$. It  induces a hermitean metric on $E$.
  
The Levi-Civita connection on $M$ uniquely lifts to an invariant connection  on $P\to M$ and induces a connection $\nabla$ on $E^{spin}$.
 We furthermore have a Clifford multiplication $c:TM\to \End(E)$. The Clifford multiplication and the connection turn $E^{spin}\to M$ into a Dirac bundle. We let
 $\Dirac_{M}^{spin}$ be the associated Dirac operator. It is a Dirac operator of degree $n$ and   satisfies the  Lichnerowicz formula
 $$(\Dirac_{M}^{spin})^{2}=\nabla^{*}\nabla+\frac{s}{4}\ ,$$
 where $s$ is the scalar curvature function on $M$. The operator $\Dirac_{M}^{spin}$ is therefore uniformly locally positive on the subsets $\{x\in M\:|\: s(x)\ge c\}$ for every positive real number $c$.
 \hB 
 \end{ex}

Let $A$ be a $G$-invariant subset of $M$.
For every $R$ in $(0,\infty)$ we let $U_{R}[A]$ denote the $R$-neighbourhood of $A$. The following definition is taken from \cite[Def. 8.2]{index}.
\begin{ddd}\label{eoihieorgregregregerge}
The subset $A$ is called a support if there exists a monotoneous sequence $(R_{i})_{i\in \nat}$ in $(0,\infty)$ tending to $\infty$ such that the inclusions $A\to U_{R_{i}}[A]$ are coarse equivalences for all $i$ in $\nat$.
\end{ddd}

 Let $\Dirac_{M}$ be an invariant Dirac operator of degree $n$ on $M$.
  We assume that there is a   $G$-invariant subset 
  $A$ of $M$ which is a support and such that $\Dirac_{M}$ is uniformly locally positive on the complement $A^{c}:=M\setminus A$.
     Then it has a coarse index class
  $$\Ind\cX(\Dirac_{M}, on(A)) \:\:\mbox{in}\:\:K\cX^{G}_{n}(\{A\})\ .$$
  In the non-equivariant situation  an  index  class with support on $A$ was first defined by Roe  \cite{roe_psc_note}.  In the equivariant case, for free proper actions, analogous index classes were introduced in Piazza-Schick \cite{MR3286895} and Zeidler \cite{MR3551834}. The index classes in these references are $K$-theory classes of versions of the Roe algebra, a $C^{*}$-algebra associated to the situation. It is a non-trivial matter to interpret them as equivariant coarse $K$-homology classes.  This problem is discussed in   \cite{index},  where we construct the index class for proper actions and as an equivariant coarse $K$-homology class as stated above.

 If $A^{\prime}$ is an  support such that $A\subseteq A^{\prime}$, then 
$\Ind\cX(\Dirac_{M}, on(A^{\prime}))$ is defined and equal to the image of  $\Ind\cX(\Dirac_{M}, on(A))$ under the natural map 
$K\cX^{G}_{n}(\{A\})\to K\cX^{G}_{n}(\{A^{\prime}\})$.

\begin{ex}\label{3ioeorgergegergre}
Below we need the following construction which allows to adapt a given Dirac operator of degree $n$ to a change of the Riemannian metric. The idea is two write the Dirac operator locally as a twisted spin Dirac operator. If we change the metric, then  we only change the spin Dirac operator part and keep the twisting fixed. Here are the details.

Let  $\Dirac_{M}$ be given as above on the Riemannian manifold $M$ with metric $g$. Then every point of $M$ admits an open neighbourhood    $U$ in  $M$ on which we can choose a spin structure $P\to U$. It induces the spin Dirac bundle $E^{spin}$ as in Example \ref{gioregrgergreg}. Let $m$ denote the dimension of $M$. 

If $n$ differs from $m$, then we use the isomorphisms $\Cl^{a}\otimes \Cl^{b}\cong \Cl^{a+b}$  for non-negative  integers $a$ and $b$ in order adjust the degrees in the following formulas.

The evaluation is  a canonical isomorphism of Dirac bundles 
$$E_{|U}\cong \left\{\begin{array}{cc} \Hom_{\Cl(TM) \otimes \Cl^{n,op}}(E^{spin}\otimes \Cl^{n-m},E_{|U})\otimes_{\Cl^{n-m}}  (E^{spin}\otimes \Cl^{n-m})& n\ge m\\
   \Hom_{\Cl(TM) \otimes \Cl^{m,op}}(E^{spin} ,E_{|U}\otimes \Cl^{m-n})\otimes_{\Cl^{m-n}} E^{spin} &   n< m\end{array} \right.$$

The first factor on the right-hand side is called the twisting bundle. 
 It is graded and carries an induced connection.

 We first explain the case $n\ge m$ in greater detail.  Note that $E^{spin}$ has a right $\Cl^{m}$-action. The tensor product $E^{spin}\otimes \Cl^{n-m}$ then has a right $\Cl^{n}$-action via the isomorphism $\Cl^{n}\cong \Cl^{m}\otimes \Cl^{n-m}$. This action and Clifford multiplication
with $TM$ on $E^{spin}$ is used in order to interpret the symbol $\Hom_{\Cl(TM) \otimes \Cl^{n,op}}$. The  left-action of $\Cl^{m-n}$ on $E^{spin}\otimes \Cl^{n-m}$ (in the first argument of $\Hom$) by multiplications   on  the second factor induces a right $\Cl^{n-m}$-action on the twisting bundle.
This is combined with the action of $\Cl^{n-m}$ by left multiplications  on the last factor in order to form the tensor product   $\otimes_{\Cl^{n-m}}$. The right $\Cl^{n}$-action on the right-hand side is induced form the right-action of $\Cl^{m} $ on $E^{spin}$ and the right-multiplication of  $\Cl^{n-m}$ on the last $\Cl^{n-m}$-factor

 We now explain the case $n<m$. The tensor product
 $E_{|U}\otimes \Cl^{m-n}$ has a right $\Cl^{m}$-action obtained from the right $\Cl^{n}$-action on $E_{|U}$ and the right-multiplication of  $\Cl^{m-n}$ on itself. This $\Cl^{m}$-action together with right $\Cl^{m}$-action on $E^{spin}$ and the Clifford multiplication of $\Cl(TM)$ is used to give meaning to the symbol $ \Hom_{\Cl(TM) \otimes \Cl^{m,op}}$. We use the left $\Cl^{m-n}$-multiplication on $\Cl^{m-n}$ in the argument of $ \Hom_{\Cl(TM) \otimes \Cl^{m,op}}$ and the right $\Cl^{m-n}$-action on $E^{Spin}$ (obtained by restriction of the $\Cl^{m}$-action) in order to interpret  the symbol $\otimes_{\Cl^{m-n}}$. There is a residual right $\Cl^{n}$-action on $E^{spin}$ which compares with the $\Cl^{n}$-action of $E_{|U}$ via the evaluation isomorphism.

 In particular, in  both cases the Clifford multiplication with tangent vectors  and the $\Cl^{n}$-action on the right-hand sides is induced from the second factor.

Assume now that $g^{\prime}$ is a second Riemannian metric on $M$.
Then we define the Dirac bundle 
$E^{\prime}\to M$ such that it is  locally  given by
$$E^{\prime}_{|U}\cong \left\{\begin{array}{cc}  \Hom_{\Cl(TM) \otimes \Cl^{n,op}}(E^{spin}\otimes \Cl^{n-m},E_{|U})\otimes_{\Cl^{n-m}}  (E^{spin,\prime}\otimes \Cl^{n-m})&   n\ge m\\  \Hom_{\Cl(TM) \otimes \Cl^{m,op}}(E^{spin} ,E_{|U}\otimes \Cl^{m-n})\otimes_{\Cl^{m-n}} E^{spin,\prime} & n< m\end{array} \right. \ .$$
where $E^{spin,\prime}$ is the spin Dirac bundle associated to the metric $g^{\prime}$. We note that $E^{\prime}$ is well-defined independent of the choice of the local spin structures.

We say that the associated Dirac operator  $\Dirac_{M}^{\prime}$ of degree $n$  is obtained from $\Dirac_{M}$ by changing the metric from $g$ to $g^{\prime}$. \hB
\end{ex}

\begin{ex}
In this example we describe the construction of the suspension $\tilde \Dirac_{M}$ of a Dirac operator $\Dirac_{M}$.

 We consider the Riemannian manifold $\tilde M:=\R\times M$ with the product metric $dr^{2}+g$, where $g$ denotes the Riemannian metric of $M$. It has an induced proper action of $G$ by isometries. 
Let $\pr:\tilde M\to M$ denote the projection. Then we can form the
graded bundle $\tilde E:=\pr^{*}E\otimes \Cl^{1}$. Using the isomorphism  $\Cl^{n}\otimes \Cl^{1}\cong \Cl^{n+1}$ we get a right $\Cl^{n+1}$-action on $\tilde E$.
Let $e$ denote the generator of $\Cl^{1}$.
The Dirac operator $\Dirac_{M}$ has a natural translation-invariant extension $\tilde M$  to $\R\times M$ of degree $n+1$
which we will denote by $\tilde  \Dirac_{ M}$.
It is given by  the (schematic) formula $\tilde \Dirac_{M}\otimes 1+\partial_{t}\otimes c(e)$, where $c(e)$ denotes left multiplication by $e$ on the $\Cl^{1}$-factor.

We use the notation $\tilde \Dirac_{M}$ in order to express the fact that this Dirac operator is obtained by a construction applied to $\Dirac_{M}$. The notation
$\Dirac_{\tilde M}$ would be used for some choice of a Dirac operator on $\tilde M$ with no precise relation with $\Dirac_{M}$.
\hB\end{ex}

\begin{rem}
%
If $M$ is a spin manifold and $\Dirac_{M}^{spin}$  is the spin Dirac operator as in Example \ref{gioregrgergreg}, then its suspension
$\tilde \Dirac^{spin}_{M}$   is isomorphic to $\Dirac^{spin}_{\R\times M}$. \hB
\end{rem}
By  \cite[Thm. 11.1]{index}   (this is the analogue and generalization of Zeidler \cite[Thm, 5.5]{MR3551834}) we have the following compatibility of coarse indices with suspension.

Let $\Dirac_{M}$ be an invariant Dirac operator on $M$ of degree $n$ which is uniformly locally positive outside of  $A$. Note that $\R\times A$ is a support in $\R\times M$.

    \begin{prop}\label{feiwofwefwefewfw}
$$ \susp(\Ind\cX(\tilde \Dirac_{ M },on(\R\times A)))=\Ind\cX(\Dirac_{M}, on(A))\ .$$
\end{prop}
 
We consider a group $G$ and let $M$ be a complete Riemannian manifold with a proper action of $G$ by isometries and an
invariant Dirac operator $\Dirac_{M}$. We assume that there is a  support
  $A$   such that $\Dirac$ is uniformly locally positive on the complement $A^{c}:=M\setminus A$. 
Let $Z$ be an open invariant  very proper (we refer to \cite[Def. 3.7]{index} for the very technical definition of this notion) subset of $M$ which is a support. 
%



We  have the big family $\{Z^{c}\}\cap Z$ in $Z$.
We assume that for every entourage $U$ of $M$ there exists an entourage $V$ of $M$ such that $U[Z\setminus V[Z^{c}]]\cap Z^{c}=\emptyset$.
In other words, for every prescribed distance $R$ in $(0,\infty)$ there exists a member of $Y$ of $\{Z^{c}\}$ such that the distance of $Z\setminus Y$ to the complement of $Z$ is bigger than than $R$.

Let $M^{\prime}$, $A^{\prime}$, $\Dirac_{M^{\prime}}$ and $Z^{\prime}$ be similar data.

We assume that there is an equivariant diffeomorphism  $i:Z\stackrel{\cong}{\to} Z^{\prime}$ which preserves the Riemannian metric.  We assume that 
$i$ also induces a morphism of $G$-bornological spaces.
  We furthermore assume that it induces a morphism 
  between the  big families  $\{A\}\cap Z$ and  $\{A^{\prime}\}\cap Z^{\prime}$, and between the big families
$\{Z^{c}\}\cap Z$ and $\{Z^{\prime,c}\}\cap Z^{\prime}$.
 We finally  assume that the isometry $Z\cong Z^{\prime}$ is
 covered by an equivariant isomorphism of Dirac operators 
$(\Dirac_{M})_{|Z}\cong (\Dirac_{M^{\prime}})_{|Z^{\prime}}$.

\begin{ddd}
We call  data  satisfying these assumptions a coarse relative index situation.
\end{ddd}

We have a diagram of horizontal fibre sequences in $G\Sp\cX$ \begin{equation}\label{cweoijowecijocwe}
\xymatrix{\Yo^{s}(\{Z^{c}\}\cap \{A\})\ar[r]&\Yo^{s}( \{A\})\ar[r]&\Yo^{s}(  \{A\})/\Yo^{s}(\{Z^{c}\}\cap \{A\})\\\Yo^{s}(Z\cap \{Z^{c}\}\cap \{A\})\ar[r]\ar[u]\ar[d] &\Yo^{s}(Z\cap \{A\})\ar[r]\ar[u]\ar[d] & \Yo^{s}(Z\cap \{A\})/\Yo^{s}(Z\cap \{Z^{c}\}\cap \{A\})\ar[u]_{\simeq}\ar[d] \\Yo^{s}(Z^{\prime}\cap \{Z^{\prime,c}\}\cap \{A^{\prime}\})\ar[r] \ar[d]&\Yo^{s}(Z^{\prime}\cap \{A^{\prime}\})\ar[r] \ar[d]& \Yo^{s}(Z^{\prime}\cap \{A^{\prime}\})/\Yo^{s}(Z^{\prime}\cap \{Z^{\prime,c}\}\cap \{A^{\prime}\}) \ar[d]^{\simeq}\\\Yo^{s}(\{Z^{\prime,c}\}\cap \{A^{\prime}\})\ar[r]&\Yo^{s}(  \{A^{\prime}\})\ar[r]&\Yo^{s}(  \{A^{\prime}\})/\Yo^{s}(\{Z^{\prime,c}\}\cap \{A^{\prime}\})} \ . 
\end{equation}
 
The lower and the upper left square are push-outs by excision. This explains the lower and upper right vertical equivalences.
The middle vertical morphisms  are induced by the isometry $i $.

\begin{ddd}
The morphism 
$$e:\Yo^{s}(\{A\})/\Yo^{s}(\{Z^{c}\}\cap \{A\})\stackrel{\simeq}{\to} \Yo^{s}( \{A^{\prime}\})/\Yo^{s}(\{Z^{\prime,c}\}\cap \{A^{\prime}\})$$
induced by the right column in \eqref{cweoijowecijocwe}
is called the excision  morphism associated to the coarse relative index situation.
\end{ddd}

We let
$\overline{\Ind\cX(\Dirac_{M}, on(A))}$   denote the image of  $ \Ind\cX(\Dirac_{M}, on(A))$ under the natural map $$K\cX^{G}(\{A\} )\to  K\cX^{G}(\{A\},\{Z^{c}\}\cap \{A\}) \ .$$ 
We define 
the class $\overline{\Ind\cX(\Dirac_{M^{\prime}}, on(A^{\prime}))}$ in $K\cX^{G}_{*}(\{A^{\prime}\},\{Z^{\prime,c}\}\cap \{A^{\prime}\})$ similarly.
The following is shown in \cite[Thm 10.4]{index}.
\begin{theorem}[Coarse Relative Index Theorem]\label{roigeorgergergerg}
We have the equality 
$$ e(\overline{\Ind\cX(\Dirac_{M}, on(A))})=  \overline{\Ind\cX(\Dirac_{M^{\prime}}, on(A^{\prime}))}\ .$$
\end{theorem}

\subsection{A refined symbol class}

We consider a group $G$ and let $M$ be a complete Riemannian manifold with a proper action of $G$ by isometries.

  We start with an invariant Dirac operator $\Dirac_{M}$ of degree $n$ on $M$ which is uniformly locally positive outside of a  support $A$ (see Definition \ref{eoihieorgregregregerge}).

We let $h\in C^{\infty}(\R)$ be a positive function.
 
\begin{ass}\label{grioregergrege} We assume
  that $h_{|(-\infty,0]}\equiv 1$ and $\lim_{t\to \infty} h(t)=\infty$.
  \end{ass}
Then we form the Riemannian manifold
$\tilde M_{h}$ given by $\R\times M$ with the warped product metric $dt^{2}+h^{2}(t)g$.
We form  the Dirac operator  $\tilde \Dirac_{h}$ on $\tilde M_{h}$ obtained from $\tilde \Dirac_{M}$ by changing the product metric to the warped product metric, see Example \ref{3ioeorgergegergre}. Since it  coincides with
$\tilde \Dirac_{M}$ on $(-\infty,0]\times M$
 the operator $\tilde \Dirac_{h}$ is uniformly locally positive on the subset
$(-\infty,0]\times A^{c}$. Its complement $$S:=([0,\infty)\times M)\cup  ((-\infty,0]\times A)$$ is a support.
Consequently we get a class
$$\Ind\cX(\tilde \Dirac_{h}, on(S) ) \:\:\mbox{in}\:\:K\cX^{G}_{*+1}(\{S\}_{\tilde M_{h}})\ .$$

We have a morphism  of $G$-bornological coarse spaces. 
 $$p:\tilde M_{h}\to \cO(M)_{-}$$ given by the identity of the underlying sets.

\begin{prop}\label{efwiowefwefewfewf} The class
$$p_{*}\Ind\cX(\tilde \Dirac_{h}, on (S)) \:\:\mbox{in}\:\:K\cX^{G}_{*}(\{S\}_{\cO(M)_{-}})$$ 
 does not depend on the choice of $h$.
 \end{prop}
 \begin{proof}
 The set $\cH$  of functions $h$ satisfying Assumption \ref{grioregergrege} is partially ordered by the relation $\le^{op}$ and filtered, where  $h \le^{op} h^{\prime}$ if and only of 
 $h^{\prime}(t)\le h(t)$ for all $t$ in $\R$.
  
  If $h,h^{\prime}$ in $\cH$ satisfy     $h\le^{op} h^{\prime} $, then the identity of the underlying sets is  a morphism $q:\tilde M_{h}\to \tilde M_{h^{\prime}}$ of $G$-bornological coarse spaces such that
  $$\xymatrix{\tilde M_{h}\ar[rr]^{q}\ar[dr]^{p}&&\tilde M_{h^{\prime}}\ar[dl]_{p^{\prime}}\\&\cO(M)_{-}&}$$
  commutes.
Since the partially ordered set $\cH$ is filtered  it suffices to show that \begin{equation}\label{foiwhofewfewfw}
q_{*}\Ind\cX(\tilde \Dirac_{h}, on (S))=\Ind\cX(\tilde \Dirac_{h^{\prime}}, on (S))
\end{equation}
   in $K\cX^{G}_{n}(\{S\}_{M_{h^{\prime}}})$.  
 The idea of the proof is to perform a deformation from $h$ to $h^{\prime}$. Technically, this deformation is done through a    geometric construction. We then    only use the two basic properties of the coarse index with support, namely locality and compatibility with suspension.

 We consider the manifold
 $\hat  M:=\R\times \R\times M$ and let $(s,t)$ denote the coordinates on the factor $\R\times \R$. We choose a real-valued smooth increasing function
 $\chi$ on $\R$ such that $\chi(s)=0$ for $s\le 0$ and $\chi(s)=1$ for $s\ge 1$. 
  We then consider the function on $\hat M$ given by 
 $$\tilde h(s,t,m):=(1-\chi(s)) h^{\prime}(t)+\chi(s) h(t)\ .$$
 We use the notation $h$ and $h^{\prime}$ also for   the functions
 $(s,t,m)\mapsto h(t)$ and $(s,t,m)\mapsto h^{\prime}(t)$ on $\hat M$.
  We have the inequalities  $$ h^{\prime}\le \tilde h\le   h\ ,$$
  We let $\hat M_{\tilde h}$ denote $\hat M$ with the warped product metric
  $ds^{2}+dt^{2}+\tilde h g$. We understand $\hat M_{h}$ and $\hat M_{h^{\prime}}$ in a similar manner.
  We then get following morphisms of $G$-bornological coarse spaces
 $$\hat M_{h}\stackrel{\hat q}{\to}  \hat M_{\tilde h}\stackrel{\hat q^{\prime}}{\to}  \hat M_{h^{\prime}} \ ,$$
 all induced by the identity of the underlying sets.
 The equality \eqref{foiwhofewfewfw} now follows from the following chain of equalities  
 \begin{eqnarray*}
 q_{*}\Ind\cX(\tilde \Dirac_{h}, on (S))&\stackrel{Prop. \ref{feiwofwefwefewfw}}{=}&\susp ( q_{*}(\Ind\cX(\hat \Dirac_{h}, on (\R\times S))))\\&=&\hat q^{\prime}_{*}(\hat q_{*}(\partial^{MV}(\Ind\cX(\hat \Dirac_{h}, on (\R\times S)))))   \\&\stackrel{!}{=}&
 \hat q^{\prime}_{*}(\hat q_{*}(\partial^{pair} (\overline{\Ind\cX(\hat  \Dirac_{h}, on (\R\times S))})))\\&\stackrel{!!}{=}& \hat q^{\prime}_{*}( \partial^{pair} (\hat q_{*}(\overline{\Ind\cX(\hat  \Dirac_{h}, on (\R\times S))})))\\&\stackrel{Thm. \ref{roigeorgergergerg}}{=}& \hat q^{\prime}_{*}( \partial^{pair} ( \overline{\Ind\cX(\hat  \Dirac_{\tilde h}, on (\R\times S))}))\\
&\stackrel{!!}{=}&
 \partial^{pair} ( \hat q^{\prime}_{*}(\overline{\Ind\cX( \hat \Dirac_{\tilde h}, on (\R\times S))})) \\
&\stackrel{Thm. \ref{roigeorgergergerg}}{=}&  \partial^{pair} (\overline{\Ind\cX(\hat  \Dirac_{ h^{\prime}}, on (\R\times S))})\\
 &\stackrel{!}{=}&
 \partial^{MV}  (\Ind\cX(\hat  \Dirac_{ h^{\prime}}, on (\R\times S))) \\
 &=&\susp  (\Ind\cX(\hat  \Dirac_{ h^{\prime}}, on (\R\times S))) \\
 &\stackrel{Prop. \ref{feiwofwefwefewfw}}{=}& \Ind\cX(\tilde \Dirac_{h^{\prime}}, on (S))\ .
 \end{eqnarray*}
 \begin{rem}\label{rioergergregerg}We note that the first application of the coarse relative index theorem \ref{roigeorgergergerg} is based on the coarse relative index situation where the  isometry
 $$([0,\infty)\times \R\times M)_{\hat M_{h}}\to ([0,\infty)\times \R\times M)_{\hat M_{\tilde h}}$$ is a morphism of $G$-bornological coarse spaces, but not necessarily an isomorphism. So the excision morphism might be non-invertible. \hB\end{rem}
   The unmarked equalities use the relation of the suspension equivalence with the Mayer-Vietoris boundary for the equivariant and coarse decomposition $((-\infty,0]\times \R\times M, [0,\infty)\times \R\times M)$ of $\hat M$, and the naturality of the Mayer-Vietoris boundary.
The equalities marked by $!$  are justified by  the relation of the Mayer-Vietoris boundary with the boundary of a pair sequence, schematically given by the right square of the commuting diagram 
$$\xymatrix{A\ar[r]\ar[d]&B\ar[d]\ar[r]^(0.4){x\mapsto \overline{ x}}&\frac{B}{A}\ar[d]^{\partial^{pair}}\\C\ar[r]&D\ar[r]^{\partial^{MV}}&\Sigma A }\ ,$$
where the left square is a push-out. 
The equalities marked by $!!$ use the naturality of the pair sequences.
For the upper part of the calculation we use the pairs
$(\hat M_{h}, (-\infty,0]\times \R\times M)$, and
$(\hat M_{\tilde h}, (-\infty,0]\times \R\times M) $,  and for the lower part we use the pairs $(\hat M_{\tilde h}, [1,\infty)\times \R\times M)$ and $(\hat M_{h^{\prime}}, [1,\infty)\times \R\times M)$.

 \end{proof}

 Recall the equivalence $c$ from \eqref{fiuhz98ifwref}.
 Let $h$ satisfy the Assumption \ref{grioregergrege}. 
\begin{ddd}\label{rgiowrggregeg}
 The class
$$\sigma(\Dirac_{M} ):= c^{-1}(p_{*}\Ind\cX(\tilde \Dirac_{h} ))\:\:\mbox{in}\:\: K\cX^{G}_{*}( \cO^{\infty}(M) )   $$
is called the symbol of $ \Dirac_{M} $.
\end{ddd}

 \begin{rem}
By   \cite[Sec. 14]{ass}  the functor $M\to K\cX^{G}(\cO^{\infty}(M))$ can be identified with the locally finite $K$-homology functor  $M\to KU^{lf}(M)$. Classically the symbols of Dirac operators are considered as locally finite $K$-homology classes on $M$.
This motivates our definition of the symbol here.  See also the discussion in  \cite[Sec. 17]{ass}.
 
 The equality asserted in Lemma \ref{wefiowoefwefwefewf} below motivates  to call $\partial^{cone}  (\sigma)$ the index of the symbol $\sigma$.
 \end{rem}

\begin{lem}\label{wefiowoefwefwefewf}
We have the equality
$$\Ind(\sigma(\Dirac_{M}))=  \Ind\cX(\Dirac_{M} )\:\:\mbox{in}\:\: K\cX^{G}_{*}(M)$$
\end{lem}
\begin{proof}
We calculate using \eqref{fiuhz98ifwref}
\begin{eqnarray*}
\Ind(\sigma(\Dirac_{M}))&\stackrel{Def. \ref{fwiowfwefewfwfw}}{=}&\partial^{cone}(\sigma(\Dirac_{M}))\\&=&
\susp(\partial^{geom}(c(\sigma(\Dirac_{M}))))\\&=&
\susp(\partial^{geom} p_{*}\Ind\cX(\tilde \Dirac_{h} ))\\&\stackrel{!}{=}&
\susp( \Ind\cX(\tilde \Dirac_{ M}) ) \\&\stackrel{Prop. \ref{feiwofwefwefewfw}}{=}&
\Ind\cX(\Dirac_{M})
\end{eqnarray*}
 The marked equality  is shown similarly as  Proposition \ref{efwiowefwefewfewf}, but in this case we can admit $h\equiv 1$. \end{proof}

%
%

In the following we define a refined symbol which takes the positivity of the Dirac operator outside of $A$ into account.

\begin{ddd}\label{eeoijowfwfwfwefwef}
We define
$$\sigma(\Dirac_{M}, on (A)):= p_{*}\Ind\cX(\tilde \Dirac_{h}, on(S)) \:\:\mbox{in}\:\:K\cX^{G}_{*}(\{S\}_{\cO(M)_{-}})\ .$$ 
\end{ddd}

Assume now that $\Dirac_{M}$ is uniformly locally positive. Then we can take
$A=\emptyset$.
 \begin{ddd}\label{iejfofjewofwefewf}
 If   $\Dirac_{M}$ is uniformly locally positive, then we define the $\rho$-invariant
 $$\rho(\Dirac_{M}):=\sigma(\Dirac_{M}, on (\emptyset))\:\:\mbox{in}\:\: K\cX^{G}_{*+1}(\cO(M))\ .$$
 \end{ddd}
 By construction, the map
 $\Yo^{s}(\cO(M))\to \cO^{\infty}(M)$  sends
 $\rho(\Dirac_{M})$ to $\sigma(\Dirac_{M})$.

\subsection{The coarse view on boundary value problems}\label{rgiorggergerge}

We consider a group $G$ and a   Riemannian $G$-manifold $W$ with boundary $M:=\partial W$.
We assume a product structure at the boundary and that the manifold $W_{\infty}$  obtained from $W$ by attaching the infinite cylinder
$M_{\infty}:=[0,\infty)\times M$ is complete.

Let $\Dirac_{W}$ be a $G$-invariant Dirac operator on $W$ with a product structure near the boundary. Then $\Dirac_{W}$ has an extension $\Dirac_{W_{\infty}}$ and a boundary reduction $\Dirac_{M}$.

The Riemannian distances  induce $G$-uniform and $G$-bornological coarse structures structures on the manifolds.  We let $M$, $W$, and $W_{\infty}$  also denote the corresponding $G$-uniform bornological coarse spaces.

\begin{ass}\label{rgioergergerg}\mbox{}
\begin{enumerate} 
\item \label{fheuiwefewfewf}We assume that the inclusion $M\to W$ is an equivalence of $G$-bornological coarse spaces.
\item We assume that $\Dirac_{M}$ is uniformly locally positive. 
\end{enumerate}
\end{ass}
The first assumption places us in the abstract situation of Section \ref{fewzfefzewifzhiufzhiwefewf}.

We use the notation introduced in Sections \ref{fewzfefzewifzhiufzhiwefewf} and
\ref{ewfwefwfwefewf}.

The second assumption implies that $\Dirac_{W_{\infty}}$ is uniformly locally positive on $M_{\infty}$. Hence we have a class (using that the subset $W$  of  $W_{\infty}$ is a support and hence nice in order to replace $\{W\}$ by $W$ in the argument of $K\cX^{G}$)
$$\Ind\cX(\Dirac_{W_{\infty}}, on (W))\:\:\mbox{in}\:\: K\cX^{G}_{*+1}(W)\ .$$
Furthermore we have a $\rho$-invariant (Definition \ref{iejfofjewofwefewf})
$$\rho(\Dirac_{M}) \:\:\mbox{in}\:\: K\cX^{G}_{*+1}(\cO(M))\ .$$
The symbol of $\Dirac_{W_{\infty}}$ (Definition \ref{rgiowrggregeg}) is a class
$$\sigma(\Dirac_{W_{\infty}})\:\:\mbox{in}\:\: K\cX^{G}_{*+2}(\cO^{\infty}(W_{\infty}))\ .$$
Definition \ref{eeoijowfwfwfwefwef} provides     a refined symbol  class
$$ \sigma(\Dirac_{W_{\infty}}, on (W)) \:\:\mbox{in}\:\:K\cX^{G}_{*+2}(Y_{\cO(W_{\infty})_{-}})\ ,$$
where we use that $Y$ is nice in order to replace $\{Y\}$ by $Y$.
By  construction we have $$i( \sigma(\Dirac_{W_{\infty}}, on (W)))=\sigma(\Dirac_{W_{\infty}})\ .$$
Consequently, \begin{equation}\label{frihjoergergreg}
\bdc(\sigma(\Dirac_{W_{\infty}}),W):=\sigma(\Dirac_{W_{\infty}}, on (W))
\:\:\mbox{in}\:\: 
 K\cX^{G}_{*+2}(Y_{\cO(W_{\infty})_{-}})
\end{equation}
 is a boundary condition (Definition \ref{oriergegegergg})  
 for the symbol $\sigma(\Dirac_{W_{\infty}})$.
 
 \begin{rem}In general, 
  the index of an elliptic  boundary value problem is determined by the interior symbol together with the boundary condition. While setting up a boundary value problem one must incorporate the interation between the interior and the boundary symbol in a suitable way. 
In our present situation the   boundary is $M$, the boundary condition is the square-integrability on $M_{\infty}$, and the Fredholm condition is implied by the local positivity of the Dirac operator on $M_{\infty}$.
 We propose that the class   $\sigma(\Dirac_{M}, on (A))$ is a suitable object which incorporates the combination of the elliptcity of the interior symbol and the boundary condition on the $K$-theoretic level. \hB
 \end{rem}

 In the following we identify the abstract index (Definition \ref{wiwogrgregregerg}) of the abstract boundary value problem for $ \sigma(\Dirac_{W_{\infty}})$  with the concrete one.
 \begin{prop}\label{rgioergegergeg}
$$\Ind( \sigma(\Dirac_{W_{\infty}}),on (W))=\Ind\cX(\Dirac_{W_{\infty}}, on(W))\ .$$
\end{prop}
\begin{proof}
By Definition \ref{wiwogrgregregerg} and \eqref{frihjoergergreg} we have
$$\Ind( \sigma(\Dirac_{W_{\infty}}),on (W))=\partial \sigma(\Dirac_{W_{\infty}}, on (W))\ .$$
The Mayer-Vietoris sequence for the decomposition $(\R\times W,[0,\infty)\times  M_{\infty})$ of $Y_{\R\otimes W_{\infty}}$ gives by flasqueness of
$([0,\infty)\times  M_{\infty})_{\R\otimes W_{\infty}}$ and
$([0,\infty)\times  M)_{\R\otimes W_{\infty}}$ an identification
$$\Yo^{s}(\R\times W)\stackrel{\simeq}{\to} \Yo^{s}(Y_{\R\otimes W_{\infty}})\ .$$
The map of Mayer-Vietoris sequences associated to the map of triples
$$(\R\otimes W,(-\infty,0]\times W,[0,\infty)\times W)\to (Y_{\R\otimes W_{\infty}},Y_{-},Y_{+})$$
induces the upper part of the  commuting diagram $$\xymatrix{\Yo^{s}(\R\otimes W)\ar[r]_{\susp}^{\simeq}\ar[d]_{g}^{\simeq}&\Sigma \Yo^{s}(W)\ar@{=}[d]\\\Yo^{s}(Y_{\R\otimes W_{\infty}})\ar[r]^{\hat \partial } &\Sigma \Yo^{s}(W)\ar@{=}[d]\\\Yo^{s}(Y_{\cO(W_{\infty})_{-}})\ar[r]^{\partial}\ar[u]^{\partial_{|Y}^{geom}}&\Sigma \Yo^{s}(W)}$$
We have $$\partial^{geom}_{|Y}  \sigma(\Dirac_{W_{\infty}},on (W))=\Ind\cX(\tilde \Dirac_{W_{\infty}},on(Y))\ ,$$
since we can again deform the warping function $h$ to the constant function $1$.
  
We now observe that the inclusion $g:\R\times W\to Y$ sends 
$\Ind\cX(\tilde \Dirac_{ W_{\infty}},on(\R\times W))$
to 
$\Ind\cX(\tilde \Dirac_{ W_{\infty}},on(Y))$.
It follows that
$$\partial  \sigma(\Dirac_{W_{\infty}},on (W))=\susp(\Ind\cX(\tilde \Dirac_{  W_{\infty}},on(\R\times W)))\stackrel{Prop. \ref{feiwofwefwefewfw}}{=}\Ind\cX(\Dirac_{W_{\infty}}, on(W))$$
 \end{proof}

 In the following we identify the abstract $\rho$-invariant (Definition \ref{fiowefoefewfewfewf}) with the concrete one. The following proposition is a version of  \cite[ Thm. 6.5.]{MR3551834}. 
 \begin{prop}\label{goergergregreg}
 We have
 $$\rho(M)=\rho(\Dirac_{M})\ .$$
 \end{prop} \begin{proof}
 We use that the equivariant coarse $K$-homology $K\cX^{G}$ is a strong equivariant coarse homology theory. Hence it factorizes over $G\Sp\cX_{wfl}$.
 
In view of Proposition \ref{fwiowoefwfewfwf} and the equality
$$\susp(\rho(\tilde \Dirac_{  M}))) \stackrel{Prop. \ref{feiwofwefwefewfw}}{=}\rho(\Dirac_{M})$$ it suffices to show that
the images of 
$ \sigma(\Dirac_{W_{\infty}},on (W))$
and 
$\rho(\tilde \Dirac_{ M})$ in the relative coarse $K$-homology group $K\cX^{G}_{*+1}(Y_{\cO(W_{\infty})_{-}},(\R\times W)_{\cO(W_{\infty})_{-}})$ coincide.  
 We can apply
 the Coarse Relative Index Theorem \ref{roigeorgergergerg}.  
The manifolds $M$ and $M^{\prime} $   in  Theorem \ref{roigeorgergergerg}
 correspond to $\widetilde{ (\R\times M)}_{h}$ and
 $\tilde W_{\infty,h}$. The relevant Dirac operators coincide on the subset 
$\R\times  M_{\infty}$ of   $\widetilde{ (\R\times M)}_{h}$ and $\tilde W_{\infty,h}$, respectively.
  The subsets $A$ and $A^{\prime}$ are
 $\R\times (-\infty,0]\times M$ and $\R\times W$. Finally, the subsets $Z$ and $Z^{\prime}$ are
 $\R\times [0,\infty)\times M$ and
 $\R\times M_{\infty}$.
 We use niceness of the subsets in order to avoid forming big families.  \end{proof}
 
 Let $\phi$ and $\psi$ be as in \eqref{hfi83i9zfh}.
 
 \begin{kor}[Piazza-Schick]
$$\phi(\rho(\Dirac_{M}))=\psi(\Ind(\Dirac_{W_{\infty}}, on(W)))\ .$$
 \end{kor}
\begin{proof}
This follows from the abstract APS-index Theorem \ref{wfiowefewfew} in combination with Proposition \ref{rgioergegergeg} and Proposition \ref{goergergregreg}.
\end{proof}

%

\begin{rem}
If one wants to exactly deduce the theorems of Piazza-Schick and Zeidler from these results above one must identify the relevant $\rho$-invariants. The invariants of Piazza-Schick  \cite{MR3286895} and Zeidler \cite[Def. 4.11]{MR3551834}
 are the analogues of our $\rho$-invariants, but they live in $K$-groups of some $C^{*}$-algebras adapted to the respective situation. It seems to be a non-trivial matter to
 provide a canonical identification of these $K$-groups with the coarse $K$-homology of the cone (see the discussion \cite[Sec. 16]{ass}). 
  \hB
\end{rem}

 \bibliographystyle{alpha}
\bibliography{dec1views}

\end{document}